\documentclass{amsart}
 \pdfoutput=1
 \usepackage{amsaddr}
\usepackage[utf8]{inputenc}
\usepackage{graphicx,amssymb,amstext,amsmath,mathtools}
\usepackage{tikz,caption,subcaption,amsthm}
\usepackage{enumitem,dsfont}
\usepackage[all,cmtip]{xy}
\usepackage{hyperref}
\usepackage{csquotes}

\newcommand{\xqedhere}[2]{%
  \rlap{\hbox to#1{\hfil\llap{\ensuremath{#2}}}}}
\makeatletter
\newcommand{\pushright}[1]{\ifmeasuring@#1\else\omit\hfill$\displaystyle#1$\fi\ignorespaces}
\newcommand{\pushleft}[1]{\ifmeasuring@#1\else\omit$\displaystyle#1$\hfill\fi\ignorespaces}
\makeatother

\theoremstyle{plain}
\newtheorem{theorem}{Theorem}[section]
\newtheorem{lemma}[theorem]{Lemma}
\newtheorem{proposition}[theorem]{Proposition}
\newtheorem{corollary}[theorem]{Corollary}
\theoremstyle{definition}

\numberwithin{equation}{section}

\author{Khadijeh Alibabaei}

\address{Centro de Matem\'{a}tica e Departamento de Matem\'{a}tica, Faculdade de Ci\^{e}ncias,
Universidade do Porto,
Porto, Portugal}
{\email{f.alibabaee@gmail.com}}
\subjclass[2000]{20E18, 20M05, 20M07, 20F10, 20K01} 	

\begin{document}
\title{Every decidable pseudovariety of abelian groups is completely  tame}

\begin{abstract}
It has been shown that the proper, non-locally finite pseudovarieties of abelian groups are not tame with respect to the canonical signature.
In this paper, we show that every decidable, proper, non-locally finite pseudovariety of abelian groups is completely tame with respect to a further enlarged implicit signature. This theorem yields as a corollary  that a pseudovariety of abelian groups is decidable if and only if it is completely tame.
\end{abstract}
\maketitle

\smallskip
\noindent \textbf{Keywords.} {relatively free profinite semigroup, pseudovariety of semigroups,  system of equations, implicit signature, completely tame, completely reducible, rational constraint, $\sigma$-full, weakly reducible.}

\section{Introduction}
 By a \emph{pseudovariety} we mean a class of semigroups which is closed under taking subsemigroups, finite direct  products and homomorphic images. A pseudovariety is said to be \emph{decidable} if there is an algorithm to test membership of a finite semigroup;   otherwise, the pseudovariety is said to be \emph{undecidable}.  Eilenberg \cite{Eilenberg:1974} established a correspondence between varieties of
rational languages and pseudovarieties of finite semigroups which translates problems in language theory into the decidability of pseudovarieties of semigroups.

In general the decidability of pseudovarieties is not preserved by many operations on pseudovarieties such as semidirect product, join and Mal'cev product (\cite{Rhodes:1997c, Albert&Baldinger&Rhodes:1992}).
Almeida and Steinberg introduced a  refined version of decidability called  \emph{tameness} \cite{Jorge&Ben:2000}. The tameness property requires the \emph{reducibility } property which is  a generalization of the notion of inevitability that Ash introduced to prove the type II conjecture of Rhodes \cite{Ash:1991}.

There are various results using tameness of pseudovarieties to establish the decidability of pseudovarieties obtained by application of the operations of semidirect product, Mal'cev product and join \cite{Almeida&Costa&Teixeira:2010, Almeida&Azevedo&Zeitoun:1997, Almeida&Costa&Zeitoun:2004}.

Also there are connections between tameness and geometry and model theory \cite{Gitik&Rips:1995, Gitik:1999b, Herwig&Lascar:1997, Almeida&Delgado:1997,Almeida&Delgado:1999}. So, it is worth finding more examples of tame pseudovarieties.

It has been established that the locally finite pseudovarieties of finite abelian groups and the pseudovariety $\mathsf{Ab}$ of all finite abelian groups are tame \cite{Jorge&&Delgado:2005}, but in \cite{Delgado&Ben&Masuda:2007}, it has been shown  that the proper, non-locally finite pseudovarieties of abelian groups are  not tame with respect to the canonical signature.
In this paper we show that such  pseudovarieties are tame with respect to a further enlarged implicit signature $\sigma$. Since the free $\sigma$-subalgebra generated by a finite alphabet   is not the free abelian group, in Section 3, we  prove that the word problem is decidable in this free $\sigma$-subalgebra, meaning that there is an algorithm to decide whether two elements of this $\sigma$-algebra represent the same element.

 In \cite{Jorge&Ben:2000}, the notions of $\sigma$-\emph{full} and  \emph{weakly $\sigma$-reducible}  pseudovariety were introduced. By   \cite[Proposition 4.5]{Jorge&Ben:2000}, every $\sigma$-full weakly  $\sigma$-reducible pseudovariety is $\sigma$-reducible.
 We use this result and we show that every decidable, proper, non-locally finite  pseudovariety of abelian groups  is $\sigma$-reducible.

\section{Preliminaries}
%\subsection{Profinite Semigroups}
A \emph{topological semigroup} is a semigroup $S$ endowed with a topology such that the basic semigroup
multiplication $S \times S \rightarrow S$ is continuous. We say that a topological semigroup $S$ is $A$-\emph{generated} if there is a mapping $\varphi : A \rightarrow S$ such that
$\varphi(A)$ generates a dense subsemigroup of $S$.

A \emph{pseudovariety} of semigroups is a class of finite semigroups closed under taking subsemigroups, homomorphic images, and direct products.
Given a pseudovariety  $\mathsf{V}$ of  semigroups, by a pro-$\mathsf{V}$ semigroup $S$ we mean a compact, zero dimensional semigroup which is residually
in $\mathsf{V}$, that is for
every two distinct points $s, t \in S$, there exists a continuous homomorphism $\varphi : S \rightarrow T$ into some
member $T\in \mathsf{V}$ such that $\varphi(s) \neq \varphi(t)$. Equivalently, a pro-$\mathsf{V}$ semigroup is a projective  limit of semigroups from $\mathsf{V}$.

  For  a finite alphabet $A$ and a pseudovariety  $\mathsf{V}$,  we denote by $\overline{\Omega}_A\mathsf{V}$ the free pro-$\mathsf{V}$ semigroup. The free pro-$\mathsf{V}$ semigroup has the universal property in the variety of pro-$\mathsf{V}$ semigroups in the sense   that,  for every mapping $\varphi : A \rightarrow S$ into a pro-$\mathsf{V}$
semigroup $S$, there exists a unique continuous homomorphism $\hat{\varphi} : \overline{\Omega}_A\mathsf{V}\rightarrow S$ such that the following
diagram commutes:
$$
\xymatrix{
A
\ar[r]^(.4)\iota
\ar[d]_\varphi
&
\overline{\Omega}_A\mathsf{V}
\ar[ld]^{\hat\varphi}
\\
S
}
$$
For an $A$-generated  pro-$\mathsf{V}$ semigroup $S$, we denote by $S^A$ the set of all functions from $A$ to $S$. To each element $w\in \overline{\Omega}_A\mathsf{V}$, we may associate an $A$-ary operation  $w_S:S^A\rightarrow S$: for every $\varphi\in S^A$, by the universal property of $\overline{\Omega}_A\mathsf{V}$, there is a unique extension $\hat\varphi:\overline{\Omega}_A\mathsf{V}\rightarrow S$.  Define $w_S(\varphi)=\hat\varphi(w)$. It is easy to see that
for every continuous
homomorphism $f : S \rightarrow T$ between pro-$\mathsf{V}$ semigroups, the
following diagram commutes:
$$
\xymatrix{
S^A
\ar[r]^{w_S}
\ar[d]_{f\circ-}
&
S
\ar[d]^{f}
\\
T^A
\ar[r]^{w_T}
&
T
}
$$
Operations with that property are called $A$-\emph{ary implicit operations}.
The element $w\in \overline{\Omega}_A\mathsf{V}$ is completely determined by the implicit operation $(w_S)_{S\in \mathsf{V}}$~\cite{Jorge:2002}. Note that, the elements of $A$ correspond to the component projects.

Let $\mathsf{S}$ be the pseudovariety of all finite semigroups. The elements of $(\overline{\Omega}_A\mathsf{S})^1$,
over arbitrary finite alphabets $A$, are called \emph{pseudowords}.
A \emph{pseudoidentity} is a formal equality of the form $u = v$ with $u, v \in \overline{\Omega}_A\mathsf{S}$ for
some finite alphabet $A$.
For a pseudovariety $\mathsf{V}$ of semigroups, we denote by $\psi_\mathsf{V}$ the unique continuous homomorphism
$\overline{\Omega}_A\mathsf{S}\rightarrow \overline{\Omega}_A\mathsf{V}$ which restricts to the identity on $A$.
We say that $\mathsf{V}$ satisfies the  pseudoidentity $u = v$ with $u, v \in \overline{\Omega}_A\mathsf{S}$ if $\psi_\mathsf{V}(u)=\psi_\mathsf{V}(v)$.

 By an \emph{implicit signature} we mean a set of implicit operations including
multiplication.
An important example is given by the canonical signature $\kappa$ consisting of the multiplication
and the unary operation $x \mapsto x^{\omega-1}$ which, to an element $s$ of a finite semigroup with $n$ elements,
associates the inverse of $s^{1+n!}$ in the cyclic subgroup generated by this power.

 Let $\sigma$ be an implicit signature. Under the natural interpretation of the
elements of $\sigma$, every profinite semigroup may be viewed as a $\sigma$-algebra in the
sense of universal algebra. The $\sigma$-subalgebra of
$\overline{\Omega}_{A}\mathsf{V}$ generated by $A$ is
denoted by ${\Omega}_{A}^{\sigma}\mathsf{V}$ and it is freely generated by $A$. We say that ${\Omega}_{A}^{\sigma}\mathsf{V}$ has  \emph{decidable word problem} if there is an algorithm to decide whether two pseudowords $u, v\in\overline{{\Omega}}_{A}\mathsf{S}$ with $\psi_{\mathsf{V}}(u),\psi_{\mathsf{V}}(v)\in {\Omega}_{A}^{\sigma}\mathsf{V}$   represent the same implicit  operation in   ${\Omega}_{A}^{\sigma}\mathsf{V}$.

%\subsection{Tameness}
For a subset $L$ of the free semigroup  ${\Omega}_A\mathsf{S}$ generated by $A$, we denote by $Cl_{\sigma}(L)$ and  $Cl(L)$  the closure of $L$ in ${\Omega}^{\sigma}_{A}\mathsf{S}$ and  $\overline{\Omega}_{A}\mathsf{S}$, respectively.

As it is mentioned in the introduction, the  property tameness requires the property reducibility. To define  reducibility we need a system of equations.
Let $X$  and $P$ be disjoint  finite sets,
whose elements will be the \emph{variables} and the \emph{parameters}  of the system, respectively.  Consider the following  system of equations:
\begin{equation}\label{eq20}
\begin{aligned}
u_i=v_i&&&&&&&& (i=1\ldots,m),
\end{aligned}
\end{equation}
with $u_i,v_i\in \overline{\Omega}_{X\cup P}\mathsf{S}$. We also fix a finite set $A$ and for every
$x \in X$, we choose a  rational subset $L_x\subseteq A^{*}$. For every parameter
$p \in P$, we associate an element $w_p\in\overline{\Omega}_{A}\mathsf{S}$. A \emph{solution} of the system \eqref{eq20} modulo $\mathsf{V}$ satisfying the constraints is  a continuous homomorphism $\delta:\overline{\Omega}_{X\cup P}\mathsf{S}\rightarrow \overline{\Omega}_{A}\mathsf{S}$
satisfying the following conditions:
\begin{enumerate}
  \item $\delta(x)\in Cl(L_x)$.
 \item $\delta(p)= w_p$.
  \item  $\mathsf{V}$ satisfies  the pseudoidentities  $\delta(u_i)=\delta(v_i)$ ($i=1\ldots,m$).
\end{enumerate}

\begin{theorem}\label{theorem4}\cite[Theorem 5.6.1]{Jorge:1994}
 Let $\mathsf{V}$ be a pseudovariety of finite semigroups.  The following conditions are equivalent for a finite system $\Sigma$ of equations with
rational constraints over the finite alphabet $A$:
\begin{itemize}
  \item $\Sigma$ has a solution in every $A$-generated finite semigroup in $\mathsf{V}$.
  \item $\Sigma$ has a solution in $\overline{\Omega}_A\mathsf{V}$.
\end{itemize}
\end{theorem}

Let $\sigma$ be an implicit signature.
Consider a system of the form \eqref{eq20}, with   constraints $L_x\subseteq A^*$ ($x\in X$)   where  $u_i$ and $v_i$ are pseudowords  in ${\Omega}^{\sigma}_{X\cup P}\mathsf{S}$ and for every parameter $p\in P$, $w_p\in  {\Omega}^{\sigma}_{A}\mathsf{S}$. Assume that this system has a solution modulo $\mathsf{V}$. A pseudovariety $\mathsf{V}$  is said to be \emph{$\sigma$-reducible}  for this system if it has a solution $\delta:\overline{\Omega}_{X\cup P}\mathsf{S}\rightarrow {\Omega}_{A}^\sigma \mathsf{S}$ modulo $\mathsf{V}$. We say that $\mathsf{V}$ is \emph{completely
$\sigma$-reducible} if it is $\sigma$-reducible for every such system.

\begin{proposition}\label{proposition1}\cite[Proposition 3.1]{Jorge&Zeitoun&Carlos:2007}
  Let $\mathsf{V}$ be a pseudovariety. If $\mathsf{V}$ is $\sigma$-reducible
respect to the  systems of equations of $\sigma$-terms without parameters, then $\mathsf{V}$ is completely
$\sigma$-reducible.
\end{proposition}

We say that a pseudovariety $\mathsf{V}$ is \emph{$\sigma$-tame} with respect to a class $\mathfrak{C}$ of systems of equations if
the following conditions hold:
\begin{itemize}
 \item  for every system of equations in $\mathfrak{C}$, the pseudovariety $\mathsf{V}$ is $\sigma$-reducible;
 \item  the word problem is decidable in $\Omega_A^{\sigma}\mathsf{V}$;
\item $\mathsf{V}$ is recursively enumerable, in the sense  that there is some Turing machine which outputs
successively  representatives of all the isomorphism classes of members of $\mathsf{V}$.
\end{itemize}
We say that $\mathsf{V}$ is \emph{completely $\sigma$-tame} if $\mathsf{V}$ is  $\sigma$-tame for every class of systems of equations. Some  important tameness results are as follows:
\begin{itemize}
  \item The pseudovariety $\mathsf{G}$ of all finite groups is $\kappa$-tame with respect to the systems of equations associated with finite
directed graphs \cite{Jorge&Ben:2000,Ash:1991}.  It follows from the results of Coulbois and Kh\'{e}lif
 that $\mathsf{G}$ is not completely $\kappa$-tame \cite{Coulbois&&Khelif}.
 \item For a prime number $p$, the pseudovariety $\mathsf{G}_p$ of all finite $p$-groups is tame  with respect to the systems of equations associated with finite directed graphs, but not $\kappa$-tame \cite{Jorge:2002}.
  \item The pseudovariety $\mathsf{Ab}$ of all finite abelian groups is completely $\kappa$-tame \cite{Jorge&&Delgado:2005}.
   \end{itemize}

Let $\mathsf{V}$ be a pseudovariety.
Consider  a system of equations (\ref{eq20}) with   constraints $L_x\subseteq A^*$  ($x\in X$)   where  $u_i$ and $v_i$ are pseudowords  in ${\Omega}^{\sigma}_{X\cup P}\mathsf{S}$ and for every $p\in P$, $\psi_{\mathsf{V}}(w_p)\in {\Omega}_{A}^\sigma \mathsf{V}$.  The pseudovariety $\mathsf{V}$ is said to be  \emph{weakly $\sigma$-reducible} with respect to this system if in the case it has a solution modulo $\mathsf{V}$, then there is a solution $\delta:\overline{\Omega}_{X\cup P}\mathsf{S}\rightarrow \overline{\Omega}_{A} \mathsf{S}$ modulo $\mathsf{V}$ which satisfies the conditions $\psi_{\mathsf{V}}(\delta(x))\in {\Omega}_{A}^\sigma \mathsf{V}$ ($x\in X$). It is obvious that if a pseudovariety $\mathsf{V}$ is $\sigma$-reducible, then it is weakly $\sigma$-reducible but the converse is not true. For a prime number $p$, the pseudovariety $\mathsf{G}_p$ of all finite $p$ groups is weakly $\kappa$-reducible but it is not $\kappa$-reducible \cite{Steinberg:1998a, Jorge&Ben:2000}.

We say that a pseudovariety $\mathsf{V}$ is $\sigma$-\emph{full} if for every rational subset
$L \subseteq A^*$, the set $\psi_\mathsf{V}(Cl_{\sigma}(L))$ is closed in ${\Omega}_{A}^{\sigma}\mathsf{V}$.

\begin{proposition}\label{proposition2}\cite[Proposition 4.5]{Jorge&Ben:2000}
Every $\sigma$-full weakly  $\sigma$-reducible pseudovariety  is $\sigma$-reducible.
\end{proposition}

A supernatural number $\pi$ is a function $\pi:\{\text{all prime numbers}\}\rightarrow \mathbb{N}\cup\{\infty\}$. We say that a supernatural number $\pi'$ divides a supernatural number $\pi$, if for every prime number $p$, we have ${\pi'}(p)\leq \pi(p)$.
The set of all supernatural numbers is a complete lattice ordered
by this relation.

 We say that $\pi$ is \emph{recursive}
if for every prime number $p$, $\pi(p)$ is computable.
For a supernatural number $\pi$, we denote by $P_\pi$ the set of all prime numbers $p$ such that $\pi(p)$ is a natural number.

  To a supernatural number $\pi$, we may associate  the pseudovariety  $\mathsf{Ab}_\pi$ of all finite abelian groups whose exponents  divide $\pi$. Conversely, to every pseudovariety $\mathsf{V}$ of abelian finite  groups, we may associate the supernatural number $\pi_{\mathsf{V}}=\mathrm{lcm}\left(\{n\mid \mathbb{Z}/n\mathbb{Z}\in \mathsf{V}\}\right)$.
\begin{lemma}\label{lemma1}\cite[Corollary 3.3]{Ben:1999}
The mapping  $\pi\mapsto \mathsf{Ab}_\pi$ and $\mathsf{V}\mapsto \pi_\mathsf{V}$ are inverse lattice isomorphisms. Furthermore, $\pi$ is recursive if and only if $\mathsf{Ab}_\pi$ has decidable membership problem.
\end{lemma}
\begin{corollary}\label{cor3}
  Let $\pi=\pi_1\pi_2$ be a supernatural number where $\pi_1$ and $\pi_2$ are relatively prime supernatural numbers. Then the pseudoidentity $u=v$ is valid in ${\mathsf{Ab}}_\pi$ if and only if the pseudoidentity $u=v$ is valid in both ${\mathsf{Ab}}_{\pi_1}$ and ${\mathsf{Ab}}_{\pi_2}$.
\end{corollary}
%\begin{proof}
 % By the preceding lemma, the greatest pseudovariety which is  contained in $\mathsf{Ab}_{\pi_1}$ and  $\mathsf{Ab}_{\pi_1}$ is trivial.
%\end{proof}

\begin{proposition}\cite[Proposition 3.4]{Ben:1999}
Let $\pi$ be a supernatural number and $A$ be a finite set. If $\pi\in  \mathbb{N}$, then $\Omega_A^{\kappa}\mathsf{Ab}_\pi=(\mathbb{Z}/\pi\mathbb{Z})^A$. Otherwise, we have $\Omega_A^{\kappa}\mathsf{Ab}_\pi=(\mathbb{Z})^A$
\end{proposition}

Let $x\in \overline{\Omega}_A\mathsf{S}$ and fix $n\in \mathbb{N}$. We denote by $x^{n^\omega}$ and $x^{n^{\omega-1}}$, the pseudowords  $\lim_{k\to \infty} x^{n^{k!}}$ and $\lim_{k\to \infty} x^{n^{k!-1}}$, respectively. Since the basic semigroup multiplication in $\overline{\Omega}_A\mathsf{S}$ is continuous, we have $$\left(x^{n^{\omega-1}}\right)^n=x^{n^{\omega}}.$$

Let $\pi$ be a supernatural number and consider the following set:
$$\sigma_\pi=\kappa\cup \left\{()^{p^{\omega-1}}\mid p\in P_\pi \right\}.$$
 In this paper we prove the following theorem.
\begin{theorem}\label{theorem5}
Let $\pi$ be a recursive supernatural number. Then  the pseudovariety $\mathsf{Ab}_\pi$ is completely $\sigma_\pi$-tame.
\end{theorem}

\begin{corollary}
  Every  pseudovariety of finite abelian groups is decidable if and only if it is completely tame.
\end{corollary}
  \begin{proof}
 If a pseudovariety $\mathsf{V}$ of abelian groups is decidable, then by Lemma \ref{lemma1}, there is a recursive supernatural number $\pi$ such that $\mathsf{V}=\mathsf{Ab}_\pi$. Hence, by the preceding theorem, $\mathsf{V}$ is completely tame.

 The converse follows from \cite{Jorge&Ben:2000}.
  \end{proof}
  From now on, we fix a recursive supernatural number $\pi$.
For simplicity,  we denote $\sigma_\pi$  by $\sigma$.
\section{The pseudovariety \texorpdfstring{$\mathsf{Ab}_{\pi}$}{TEXT} is \texorpdfstring{$\sigma$}{TEX}-tame}
\subsection{The word problem is decidable in \texorpdfstring{$\Omega_A^{\sigma}\mathsf{Ab}_{\pi}$}{TEXT}}

We denote by $\widehat{\mathbb{Z} }_{\pi}$ and $\mathbb{Z}_\pi^{\sigma}$, the profinite group $\overline{\Omega}_{1} \mathsf{Ab}_\pi$ and the free $\sigma$-algebra ${\Omega}_1^{\sigma} \mathsf{Ab}_\pi$, respectively. Here, we take $\overline{\Omega}_{A} \mathsf{Ab}_\pi$ as an additive topological  group and therefore, the image of $\sigma$ under the projection $\psi_{\mathsf{Ab}_\pi}$ is the following implicit signature:
$$\left\{+ \ ,\  -\right\}\cup \left\{{p^{\omega-1}()}\mid p\in P_\pi\right\}.$$
 The elements of $\widehat{\mathbb{Z}}_\pi$ are called \emph{pseudonumbers}.

\begin{lemma}\label{theorem1}
For every implicit signature $\mathfrak{s}$ and every supernatural number~$\pi$,  the $\mathfrak{s}$-algebras $\Omega^\mathfrak{s}_{\{a_1,\ldots,a_r\}}\mathsf{Ab}_\pi$ and $\displaystyle\prod_{i=1}^r\Omega^\mathfrak{s}_{\{a_i\}}\mathsf{Ab}_\pi$ are isomorphic.
\end{lemma}
\begin{proof}
The proof is essentially the same as the proof of \cite[Lemma 2.1]{Jorge&&Delgado:2005}.
\end{proof}

By the preceding lemma, we just need to prove that  the word problem is decidable in $\mathbb{Z}_\pi^{\sigma}$. From now on, we take $A=\{1\}$ and we prove that $\mathbb{Z}_\pi^{\sigma}$  has decidable word problem. In the following lemma, we find a simple form for the pseudonumbers  in $\mathbb{Z}_\pi^{\sigma}$.

\begin{lemma}\label{lemma50}
  Every pseudonumber  in $\mathbb{Z}_\pi^{\sigma}$ can be written in the following form:
\begin{equation}\label{eq12}
\begin{aligned}
 & n_1^{\omega-k_1}a_1+ n_2^{\omega-k_2}a_2+\cdots+n_m^{\omega-k_m}a_m+a_0,
\end{aligned}
\end{equation}
 with the  following conditions:
\begin{enumerate}
  \item for every $j$  $(1\leq j\leq m)$, $n_i$ and $k_j$ are  natural numbers;
  \item for every $i$ $(0\leq i\leq m)$, $a_{i}$ is in $\mathbb{Z}$;
  \item for every prime number $p$ dividing $n_j$ $(1\leq j\leq m)$, $p$ lies in $P_\pi$.
\end{enumerate}
\end{lemma}

\begin{proof}
Since every pseudonumber  in $\mathbb{Z}_\pi^{\sigma}$ is constructed from  $1$ using a finite number of times the operations  in $\sigma$, the pseudonumbers $(p^{\omega-1})^k$ are in $\mathbb{Z}_\pi^{\sigma}$ ($p\in P_\pi$ and $k\in \mathbb{N}$).

It is easy to prove the following  properties of the operations in $\sigma$:
\begin{itemize}
  \item $(p^{\omega-1})^k=p^{\omega-k}$ ($k\in\mathbb{N}$);
  \item $p^{\omega-k_1}p^{\omega-k_2}=p^{\omega-(k_1+k_2)}$ ($k_1,k_2\in\mathbb{N}$);
\item for every $k_1,k_2,\ldots,k_m\in \mathbb{N}$ with the property  $k_1\geq k_2\geq \ldots\geq k_m$ and for every $p_1,\ldots, p_m\in P_\pi$, we have
\begin{align*}
  p_1^{\omega-k_1}p_2^{\omega-k_2}\ldots p_m^{\omega-k_m}
    &=\left(p_2^{k_1-k_2}p_3^{k_1-k_3}\ldots p_{m}^{k_1-k_{m}}\right)\left(p_1p_2\ldots p_m\right)^{\omega-k_1}.
\end{align*}
 \end{itemize}

As $\mathbb{Z}_\pi^{\sigma}$ is a commutative algebra, by the above properties  every pseudonumber  in $\mathbb{Z}_\pi^{\sigma}$  can be written in the following form:

\begin{equation}\label{eq9}
\begin{aligned}
 &\left(n_1^{\omega-k_1}a_{1,k_1}+ n_1^{\omega-(k_1-1)}a_{1,(k_1-1)}+\cdots+n_1^{\omega-1}a_{1,1}\right)+\\
&\left( n_2^{\omega-k_2}a_{2,k_2}+ n_2^{\omega-(k_2-1)}a_{2,(k_2-1)}+\cdots+n_2^{\omega-1}a_{2,1}\right)+\cdots+\\
&\left(n_m^{\omega-k_m}a_{1,k_m}+n_m^{\omega-(k_m-1)}a_{1,(k_m-1)}+\cdots+n_m^{\omega-1}a_{m,1}\right)+a_0,
\end{aligned}
\end{equation}

with the following conditions:
\begin{enumerate}
  \item for every $j$ ($1\leq j\leq m$), $n_j$ and $k_j$ are  natural numbers;
  \item for  every $i$ and $j$, $a_{j,{k_j-t_j}}$ and $a_0$ are in $\mathbb{Z}$ ($1\leq j\leq m$ and $0\leq t_j\leq k_j-1$);
  \item for every prime number $p$ dividing  $n_i$ ($1\leq i\leq m$), $p$ lies in $P_\pi$.
\end{enumerate}
We can still  do more factorizations as follows:
\begin{equation}
\begin{aligned}
 &n_1^{\omega-k_1}\left(a_{1,k_1}+ n_1a_{1,(k_1-1)}+\cdots+n_1^{k_1-1}a_{1,1}\right)+\\
& n_2^{\omega-k_2}\left(a_{2,k_2}+ n_2a_{2,(k_2-1)}+\cdots+n_2^{k_2-1}a_{2,1}\right)+\cdots+\\
&n_m^{\omega-k_m}\left(a_{1,k_m}+n_ma_{1,(k_m-1)}+\cdots+n_m^{k_m-1}a_{m,1}\right)+a_0.
\end{aligned}
\end{equation}
Hence, every pseudonumber  in $\mathbb{Z}_\pi^{\sigma}$ can be written in  the following form:
\begin{equation*}
\begin{aligned}
 & n_1^{\omega-k_1}a_1+ n_2^{\omega-k_2}a_2+\cdots+n_m^{\omega-k_m}a_m+a_0,
\end{aligned}
\end{equation*}
 with the following conditions:
\begin{enumerate}
 \item for every $j$  $(1\leq j\leq m)$, $n_j$ and  $k_j$ are natural numbers;
  \item for every $i$   $(0\leq i\leq m)$, $a_{i}$ is in $\mathbb{Z}$;
  \item for every prime number $p$ dividing $n_j$ ($1\leq j\leq m$), $p$ lies in $P_\pi$.\qedhere
\end{enumerate}
\end{proof}

 Note that,  as  $\widehat{\mathbb{Z}}_\pi$  is the projective limit of the finite rings $\mathbb{Z}/n\mathbb{Z}$ ($n$ divides $\pi$), the profinite group $\widehat{\mathbb{Z}}_\pi$ inherit  a natural structure of a ring. Using the preceding lemma, it is easy to see that ${\mathbb{Z}}^\sigma_\pi$ is a subring of the profinite ring $\widehat{\mathbb{Z}}_\pi$.

\begin{lemma}\label{lemma45}
Let $\pi$ be a supernatural number and let $p$ be a prime number such that  $\gcd(p^\infty,\pi)=p^m$ with $m\in \mathbb{N}$. Let $p^m, p^{\omega+m}\in {\mathbb{Z}}_\pi^\sigma$. Then the pseudoidentity $p^{\omega+m}=p^m$  is valid in $\mathsf{Ab}_\pi$.
\end{lemma}
\begin{proof}
If $m=0$, then by the Euler congruence theorem, the pseudoidentity $p^{\omega}=1$ is valid in $\mathsf{Ab}_\pi$.

Let $m>0$ and  $\pi=p^m\pi'$. Since $\gcd(p^m,\pi')=1$, the pseudoidentity  $p^m(p^\omega)=p^m$ is valid in  $\mathsf{Ab}_{\pi'}$. On the other hand, the pseudovariety $\mathsf{Ab}_{p^m}$ satisfies the pseudoidentities $p^m(p^\omega)=0=p^m$. Hence , by Corollary \ref{cor3}, the pseudoidentity $p^m(p^\omega)=p^m$ is valid in $\mathsf{Ab}_{\pi}$.
\end{proof}

 By the  following lemmas and Corollary \ref{cor3}, we can reduce the word problem in  $\mathbb{Z}_\pi^{\sigma}$ to the word problems in the free abelian  group $\mathbb{Z}$ and the finite abelian groups. We know that the word problem is decidable in the finitely generated free abelian groups and the finite abelian groups.

\begin{lemma}\label{lemma46}
  Let $\pi$ be a recursive supernatural number and $u\in \mathbb{Z}_\pi^{\sigma}$. There is a computable $c_u\in\mathbb{N}$ such that $c_uu\in \mathbb{Z}$ and for every prime number $p$ dividing $c_u$, $p$ belongs to $P_\pi$.
\end{lemma}

 \begin{proof}
      By Lemma \ref{lemma50}, $u$ can be written in the following form:
      \begin{equation}
      \begin{aligned}
      u&= n_1^{\omega-k_1}a_1+ n_2^{\omega-k_2}a_2+\cdots+n_s^{\omega-k_s}a_s+a_0,
      \end{aligned}
      \end{equation}
     with the following conditions:
     \begin{enumerate}
     \item for every $j$  $(1\leq j\leq m)$, $n_j$  and $k_j$ are  natural numbers;
  \item for every $i$   $(0\leq i\leq m)$, $a_{i}$ is in $\mathbb{Z}$;
       \item for every prime number $p$ dividing  $n_j$ ($1\leq j\leq s$), $p$ lies in $P_\pi$.
     \end{enumerate}
   Let $\{p_1,\ldots,p_l\}\subseteq P_\pi$ be the set of all prime  number divisors of the natural number $n_1\ldots n_s$ and
   let $\gcd(p_i^{\infty},\pi)=p_i^{\beta_i} $. Consider  $\beta=\max\{\beta_1,\ldots,\beta_l\}$ and
   $$c_u= n_1^{k_1+\beta}\ldots n_s^{k_s+\beta}.$$
   Since $\pi$ is a recursive supernatural number, $\beta$ is a computable natural number.

    We show that $c_uu$ belongs to $\mathbb{Z}$.
With the choice of the set $\{p_1,\ldots,p_l\}$, there are $e_{i,k}\in \mathbb{N}$ such that $n_i=p_1^{e_{i,1}}\ldots p_l^{e_{i,l}}$ ($1\leq i\leq s$). Since, if $e_{i,j}$ is a nonzero natural number,  then $e_{i,j}\beta\geq \beta_j$, by Lemma \ref{lemma45} the following pseudoidentities are valid in $\mathsf{Ab}_\pi$:
   \begin{equation}\label{eq14}
     p_j^{\omega+e_{i,j}\beta}=p_j^{e_{i,j}\beta} \ \ \ (1\leq j\leq l).
   \end{equation}
      So, the following pseudoidentities are  valid in $\mathsf{Ab}_\pi$:
   \begin{equation*}
     \begin{aligned}
      c_un_i^{\omega-k_i}&= n_1^{k_1+\beta}\ldots n_s^{k_s+\beta}n_i^{\omega-k_i}\\
        &=\left(n_1^{k_1+\beta}\ldots n_{i-1}^{k_{i-1}+\beta}n_{i+1}^{k_{i+1}+\beta}\ldots  n_s^{k_s+\beta}\right)n_i^{k_i+\beta}n_i^{\omega-k_i}\\
       &=\left(n_1^{k_1+\beta}\ldots n_{i-1}^{k_{i-1}+\beta}n_{i+1}^{k_{i+1}+\beta}\ldots  n_s^{k_s+\beta}\right)n_i^{\omega+\beta}\\
        &=\left(n_1^{k_1+\beta}\ldots n_{i-1}^{k_{i-1}+\beta}n_{i+1}^{k_{i+1}+\beta}\ldots  n_s^{k_s+\beta}\right)\left(p_1^{e_{i,1}}\ldots p_l^{e_{i,l}}\right)^{\omega+\beta}\\
        &=\left(n_1^{k_1+\beta}\ldots n_{i-1}^{k_{i-1}+\beta}n_{i+1}^{k_{i+1}+\beta}\ldots  n_s^{k_s+\beta}\right)\left(p_1^{\xi_1(\omega+\beta e_{i,1})}\ldots p_l^{\xi_l(\omega+\beta e_{i,l})}\right)\\
        & \stackrel{\eqref{eq14}}{=}\left(n_1^{k_1+\beta}\ldots n_{i-1}^{k_{i-1}+\beta}n_{i+1}^{k_{i+1}+\beta}\ldots  n_s^{k_s+\beta}\right)\left(p_1^{\xi_1(\beta e_{i,1})}\ldots p_l^{\xi_l(\beta e_{i,l})}\right)\in\mathbb{Z},
     \end{aligned}
    \end{equation*}
    where if we have $e_{i,j}>0$, then $\xi_j=1$, otherwise $\xi_j=0$. Hence, $c_uu$ lies in~$\mathbb{Z}$.
 \end{proof}

\begin{lemma}\label{lemma38}
    Let $\pi$ be a supernatural number and $n\in \mathbb{N}$ be such that $n$ and $\pi $ are relatively prime. Let $u,v\in \widehat{\mathbb{Z}}_\pi$. Then the pseudoidentity $nu=nv$ is valid  in $\mathsf{Ab}_\pi$ if and only if the pseudoidentity $u=v$ is valid in $\mathsf{Ab}_\pi$.
\end{lemma}
  \begin{proof}
   If the pseudoidentity $u=v$ is valid in $\mathsf{Ab}_\pi$, then the pseudoidentity $nu=nv$ is valid in $\mathsf{Ab}_\pi$.

    Conversely, suppose that the pseudoidentity $nu=nv$ is valid  in $\mathsf{Ab}_\pi$. Since $\gcd(n,\pi)=1$, by Lemma \ref{lemma45}, the pseudoidentity $n^{\omega}=1$ is valid in $\mathsf{Ab}_\pi$. Hence, $\mathsf{Ab}_\pi$ satisfies the following pseudoidentities:
    $$u=n^\omega u=n^{\omega-1}(nu)=n^{\omega-1}nv=n^{\omega}v=v.$$
    So,  the pseudoidentity $u=v$ is valid in $\mathsf{Ab}_\pi$.
  \end{proof}

Now we prove the main result  of this section.
\begin{theorem}
Let $\pi$ be a recursive supernatural number. Then the word problem is decidable  in ${\mathbb{Z}}^{\sigma}_\pi$.
\end{theorem}
\begin{proof}
Let $u,v\in \mathbb{Z}_\pi^{\sigma}$. By Lemma \ref{lemma46}, there are computable $c_u,c_v\in\mathbb{N}$ such that $c_uu,c_vv\in \mathbb{Z}$ and for every prime number $p$ dividing $c_uc_v$, $p$ lies in $P_\pi$.
Let  $\{p_1,\dots,p_l\}\subseteq P_\pi$ be the set of all prime number divisors of the natural number  $c_uc_v$.
There are $\beta_i\in \mathbb{N}$ such that $\pi=p_1^{\beta_1}\ldots p_l^{\beta_l}\pi'$ and  $\gcd(\pi',p_1\ldots p_l)=1$.

By  Corollary \ref{cor3}, the pseudoidentity $u=v$ is valid in $\mathsf{Ab}_\pi$ if and only if  the pseudoidentity $u=v$ is valid in both $\mathsf{Ab}_{\pi'}$ and $\mathsf{Ab}_{(p_1^{\beta_1}\ldots p_l^{\beta_l})}$.

First  we  show that it is decidable whether the pseudoidentity $u=v$ is valid in $\mathsf{Ab}_{\pi'}$. We have $\gcd(c_uc_v,\pi')=1$. Hence, by Lemma \ref{lemma38}, the pseudoidentity $u=v$ is valid  in $\mathsf{Ab}_{\pi'}$ if and only if the pseudoidentity
 $(c_uc_v)u=(c_uc_v)v$ is valid in $\mathsf{Ab}_{\pi'}$. Since $c_vc_uu$ and $c_uc_vv$ are in $\mathbb{Z}$ and $\mathbb{Z}\subseteq \widehat{\mathbb{Z}}_{\pi'}$, the pseudoidentity $u=v$ is valid in $\mathsf{Ab}_{\pi'}$ if and only if
  $c_uc_vu-c_uc_vv=0$. So, it is decidable whether the pseudoidentity  $u=v$ is valid in $\mathsf{Ab}_{\pi'}$.

 On the other hand, the pseudoidentity $u=v$ holds in  $\mathsf{Ab}_{(p_1^{\beta_1}\ldots p_l^{\beta_l})}$, if and only if  for the
 natural projection
\begin{equation*}
\begin{aligned}
\varphi:\widehat{\mathbb{Z}}_{\pi}&\longrightarrow& \mathbb{Z}/(p_1^{\beta_1}\ldots p_l^{\beta_l})\mathbb{Z}\\
1&\mapsto& 1+(p_1^{\beta_1}\ldots p_l^{\beta_l})\mathbb{Z}
\end{aligned}
\end{equation*}
the equality $\varphi(u)=\varphi(v)$ holds.
We  show that $\varphi(u)$ and $\varphi(v)$ are computable, hence, it is decidable whether $\varphi(u)=\varphi(v)$.

To show that $\varphi(u)$  and $\varphi(v)$ are  computable, it is enough to show that for every prime number $p\in P_\pi$, $\varphi(p^{\omega-1})$ is computable.

We consider  the two following cases:
\begin{enumerate}
  \item The prime number $p$ does not belong to the set $\{p_1,\ldots,p_l\}$. Since $p\in P_\pi$ and  $\pi$ is a recursive supernatural number, there is a computable $\beta\in\mathbb{N}$ such that $\gcd(\pi,p^{\infty})=p^{\beta}$.   We have
      $$\gcd(p^{\beta+1}, p_1^{\beta_1}\ldots p_l^{\beta_l})=1.$$ So, by the Euclidean algorithm, one may compute  $s,t\in \mathbb{Z}$ such that $sp^{\beta+1}+ t p_1^{\beta_1}\ldots p_l^{\beta_l}=1$. Hence, the following pseudoidentities are valid in $\mathsf{Ab}_\pi$:
\begin{equation*}
\begin{aligned}
 p^{\omega-1}&=p^{\omega-1}\left(sp^{\beta+1}+ tp_1^{\beta_1}\ldots p_l^{\beta_l}\right)\\
&=sp^{\omega+\beta}+tp^{\omega-1} p_1^{\beta_1}\ldots p_l^{\beta_l}\\
&= sp^{\beta}+tp^{\omega-1} p_1^{\beta_1}\ldots p_l^{\beta_l},
\end{aligned}
\end{equation*}
the last pseudoidentity  follows from Lemma \ref{lemma45}. So, we have
$$\varphi(p^{\omega-1})=sp^{\beta}+\left(p_1^{\beta_1}\ldots p_l^{\beta_l}\right)\mathbb{Z}.$$

 \item The prime number $p$ belongs to $\{p_1,\ldots,p_l\}$. Suppose that $p=p_i$. Since $\gcd(p_i^{\beta_i+1}, p_1^{\beta_1}\ldots p_{i-1}^{\beta_{i-1}}p_{i+1}^{\beta_{i+1}}\ldots p_{l}^{\beta_{l}})=1$, one may compute  $s,t\in \mathbb{Z}$ such that
$$sp_i^{\beta_i+1}+tp_1^{\beta_1}\ldots  p_{i-1}^{\beta_{i-1}}p_{i+1}^{\beta_{i+1}}\ldots p_{l}^{\beta_{l}}=1.$$
Hence, the following pseudoidentities hold in $\mathsf{Ab}_\pi$:
\begin{equation*}
\begin{aligned}
 p_i^{\omega-1}&=p_i^{\omega-1}\left(sp_i^{\beta_i+1}+ tp_1^{\beta_1}\ldots  p_{i-1}^{\beta_{i-1}}p_{i+1}^{\beta_{i+1}}\ldots p_{l}^{\beta_{l}}\right)\\
 &=sp_i^{\omega+\beta_i}+tp_i^{\omega-1} \left(p_1^{\beta_1}\ldots  p_{i-1}^{\beta_{i-1}}p_{i+1}^{\beta_{i+1}}\ldots p_{l}^{\beta_{l}}\right)\\
&=sp_i^{\beta_i}+tp_i^{\omega-1}\left(p_1^{\beta_1}\ldots  p_{i-1}^{\beta_{i-1}}p_{i+1}^{\beta_{i+1}}\ldots p_{l}^{\beta_{l}}\right)\\
&=sp_i^{\beta_i}+tp_i^{\omega-\beta_i-1}\left(p_1^{\beta_1}\ldots  p_{i-1}^{\beta_{i-1}}p_i^{\beta_i}p_{i+1}^{\beta_{i+1}}\ldots p_{l}^{\beta_{l}}\right).
\end{aligned}
\end{equation*}
Therefore, we have $\varphi(p_i^{\omega-1})=sp_i^{\beta_i}+\left(p_1^{\beta_1}\ldots p_l^{\beta_l}\right)\mathbb{Z}$.
\end{enumerate}
Since $\varphi$ is a homomorphism,  $\varphi(u)$ and $\varphi(v)$ are computable, so, it is decidable whether the equality  $\varphi(u)=\varphi(v)$ holds in $\mathbb{Z}/(p_1^{\beta_1}\ldots p_l^{\beta_l})\mathbb{Z}$.
\end{proof}

\begin{corollary}\label{theorem2}
Let $\pi$ be a recursive supernatural number. The word problem in the $\sigma$-algebra  $\Omega_A^\sigma \mathsf{Ab}_\pi$ is decidable.
\end{corollary}

\subsection{Systems of equations over \texorpdfstring{$\widehat{\mathbb{Z}}_{\pi}$}{TEXT}
  and \texorpdfstring{${\mathbb{Z}}^\sigma$} {TEXT}}

In this section, we show that if  a linear system of equations with coefficients in $\mathbb{Z}_\pi^\sigma$ has a solution in $\widehat{\mathbb{Z}}_\pi$, then it has a solution in $\mathbb{Z}_\pi^\sigma$. In the last section, we use this result to show that the pseudovariety $\mathsf{Ab}_\pi$ is  $\sigma$-reducible.

By  elementary Number Theory,   we know that  a congruence
$$ax\equiv b \pmod n$$
with $a,b,n\in \mathbb{Z}$ has a solution in $\mathbb{Z}$ if and only if  $\gcd(a,n)$ divides $b$. We generalize this result as follows:

\begin{lemma}\label{lemma39}
 The  equation $ax=b$ with $a,b\in \mathbb{Z}$ has a solution in $\widehat{\mathbb{Z}}_\pi$ if and only if  $d=\gcd(a,\pi)$ divides $b$ in $\mathbb{Z}$.
\end{lemma}
\begin{proof}

First suppose that the equation $ax=b$ with $a,b\in \mathbb{Z}$ has a solution in $\widehat{\mathbb{Z}}_\pi$.   If $a=0$, then we have $b=0$.

 Suppose that $a\neq0$.   Let $\varphi :\widehat{\mathbb{Z}}_\pi\rightarrow \mathbb{Z}/d\mathbb{Z}\in \mathsf{Ab}_\pi$ be defined by $\varphi(1)=1+d\mathbb{Z}$. Since $d$ divides  $a$, we have $\varphi(ax)=0$. Hence, the equalities   $b+d\mathbb{Z}=\varphi(b)=\varphi(ax)=0$ hold. Then  $d$ divides $b$ in $\mathbb{Z}$.

Conversely, suppose that there is $k\in\mathbb{Z}$ such that $b=dk$. Let $m\in \mathbb{N}$ be such that $m$ divides $\pi$. We have $\gcd(m,a)\mid d$. So, by the Euclidean algorithm, we may compute  $s,t\in\mathbb{Z}$ such that $sm+ta=d$. Hence, the following equalities hold
$$b=dk=ksm+kta.$$
Now, let $\varphi:\widehat{\mathbb{Z}}_\pi\rightarrow \mathbb{Z}/m\mathbb{Z}\in \mathsf{Ab}_\pi$ be a homomorphism. Then the equality $\varphi(b)=\varphi(kt)\varphi(a)$ holds. Therefore, the equation $ax=b$ has a solution in every finite abelian group in $\mathsf{Ab}_\pi$. By Theorem \ref{theorem4}, $ax=b$ has a solution in $\widehat{\mathbb{Z}}_\pi$.
\end{proof}

Now by using the preceding lemma, we show that if a single equation with a single variable and coefficients in $\mathbb{Z}^\sigma_\pi$ has a solution in $\widehat{\mathbb{Z}}_\pi$, then  it has a solution in $\mathbb{Z}_\pi^\sigma$. Then  we generalize this result to a linear systems of equations with coefficients in $\mathbb{Z}^\sigma_\pi$.
\begin{theorem}
   Let $u,v\in \mathbb{Z}_\pi^{\sigma}$. If the equation $ux=v$ has a solution in $\widehat{\mathbb{Z}}_{\pi}$, then it has a solution in $\mathbb{Z}_\pi^{\sigma}$.
  \end{theorem}
\begin{proof}
  If $u=0$, then we have $v=0$. So, we assume that $u\neq 0$.

 Since $u,v\in \mathbb{Z}_\pi^{\sigma}$, by Lemma \ref{lemma46}, there are $c_u,c_v\in\mathbb{N}$ such that $c_uu,c_vv\in\mathbb{Z}$  and for every prime number $p$ dividing  $c_uc_v$, $p$ lies in $P_\pi$. Let $\{p_1,\ldots,p_l\}\subseteq P_\pi$ be the set of all prime number  divisors of the natural number $c_uc_v$. We may decompose the integer number $c_uu$ as follows:
\begin{equation}
c_uu=p_1^{\xi_1}\ldots p_l^{\xi_l}p_{l+1}^{\xi_{l+1}}\ldots p_{l+b}^{\xi_{l+b}}dk,
\end{equation}
with the following properties:
\begin{itemize}
\item for every $i$ ($1\leq i\leq l$), $\xi_i\in \mathbb{N}$  and for every $j$ ($l+1\leq j\leq l+b$), $\xi_j$ are  nonzero natural numbers;
   \item $p_i\in P_\pi$ are distinct prime numbers. Let $\beta_i\in \mathbb{N}$ be such that $$\gcd(p_i^{\infty},\pi)=p_i^{\beta_i}\ \ \  (1\leq i\leq l+b);$$
    \item for every prime  number $p$ dividing $d$, $p^{\infty}$ divides $\pi$. It follows that $d$ divides $\pi$;
    \item $\gcd(k,\pi)=1$.
\end{itemize}
Note that for every prime number $p$ dividing $(c_uu)/d$, $p$ lies in $P_\pi$.
The supernatural number $\pi$ can be  decomposed  into  $p_1^{\beta_1}\ldots p_{l+b}^{\beta_{l+b}}\pi'$ with the condition $\gcd(\pi', p_1\ldots p_{l+b}) =1$.  So, we have
 \begin{equation}\label{eq6}
 \begin{aligned}
    \gcd({c_uu},\pi')&= \gcd(dk,\pi')=d,
  \end{aligned}
 \end{equation}
 \begin{equation}\label{eq7}
 \begin{aligned}
  \gcd(\frac{c_uu}{d},\pi')=1.
  \end{aligned}
  \end{equation}

 Let $\varphi:\mathbb{Z}_\pi\rightarrow \mathbb{Z}/(p_1^{\beta_1}\ldots p_{l+b}^{\beta_{l+b}})\mathbb{Z}$ be the natural projection defined by $\varphi(1)=1+p_1^{\beta_1}\ldots p_{l+b}^{\beta_{l+b}}\mathbb{Z}$. Since the equation $ux=v$ has a solution in~$\widehat{\mathbb{Z}}_{\pi}$, the congruence  $$\varphi(u)x\equiv\varphi(v) \ \ \  \pmod{{p_1^{\beta_1}\ldots p_{l+b}^{\beta_{l+b}}}}$$ has a solution in $\mathbb{Z}$.
Let $x_1\in\mathbb{Z}$ be a solution of this congruence.

   Now we find  $x_2\in \mathbb{Z}_\pi^{\sigma}$ such that  $x_2$ is a solution of the equation $ux=v$ in $\widehat{\mathbb{Z}}_{\pi'}$.
Since the equation $ux=v$ has a solution in $\widehat{\mathbb{Z}}_{\pi}$, this equation has a  solution in~$\widehat{\mathbb{Z}}_{\pi'}$. Hence, $c_uc_vux=c_uc_vv$ has a solution in $\widehat{\mathbb{Z}}_{\pi'}$. So,
 by Lemma \ref{lemma39}, there is $k\in \mathbb{Z}$ such that $c_uc_vv=gcd(c_vc_uu,\pi')k$. As $\gcd(c_u,\pi')=\gcd(c_v,\pi')=1$, we have  $d=\gcd(c_vc_uu,\pi')$ and the pseudoidentity $(c_uc_v)^{\omega}=1$ is valid  in $\mathsf{Ab}_{\pi'}$. Hence, the following pseudoidentities hold in $\mathsf{Ab}_{\pi'}$:
\begin{equation}\label{eq10}
v=(c_uc_v)^\omega v=(c_uc_v)^{\omega-1}(c_uc_vv)=(c_uc_v)^{\omega-1}dk.
\end{equation}
Let
\begin{equation*}
    \begin{aligned}
    x_2&=c_uk\left(\frac{c_uu}{d}\right)^{\omega-1}(c_uc_v)^{\omega-1}\in \mathbb{Z}_\pi^{\sigma}.
    \end{aligned}
  \end{equation*}
By \eqref{eq7},  the pseudoidentity $(c_uu/d)^\omega=1$ is valid  in $\mathsf{Ab}_{\pi'}$. Hence, the following pseudoidentities are valid in $\mathsf{Ab}_{\pi'}$,
\begin{equation}\label{eq4}
\begin{aligned}
ux_2&=uc_u\left(\frac{c_uu}{d}\right)^{\omega-1}(c_uc_v)^{\omega-1}k\\
&=d\left(\frac{c_uu}{d}\right)^{\omega}(c_uc_v)^{\omega-1}k=d(c_uc_v)^{\omega-1}k\stackrel{\eqref{eq10}}{=}v.
\end{aligned}
\end{equation}
So, $x_2\in\mathbb{Z}_\pi^{\sigma}$ is a solution of the equation $ux=v$ in $\widehat{\mathbb{Z}}_{\pi'}$.

Let
\begin{equation*}
  w=x_1+\left(p_1\ldots p_{l+b}\right)^{\omega}\left(x_2-x_1\right)\in \mathbb{Z}_\pi^{\sigma}.
\end{equation*}
We claim that $w$ is a solution of the equation $ux=v$ in $\mathbb{Z}_\pi^{\sigma}$. By Corollary \ref{cor3}, it is enough to show that the pseudoidentity $uw=v$
is valid in both $\mathsf{Ab}_{\pi'}$ and $\mathsf{Ab}_{(p_1^{\beta_1}\ldots p_{l+b}^{\beta_{l+b}})}$.

   It remains to  establish  the claim.  Since $\gcd(\pi',p_1\ldots p_{l+b})=1$, the pseudoidentity $\left(p_1\ldots p_{l+b}\right)^{\omega}=1$ is valid in $\mathsf{Ab}_{\pi'}$. Hence, the following pseudoidentities are valid in $\mathsf{Ab}_{\pi'}$:
 \begin{equation*}
   \begin{aligned}
   uw&=u\left(x_1+\left(p_1\ldots p_{l+b}\right)^{\omega}(x_2-x_1)\right)\\
   &=u\left(x_1+(x_2-x_1)\right)=ux_2\stackrel{\eqref{eq4}}{=}v.
   \end{aligned}
 \end{equation*}
 So, the pseudoidentity $uw=v$ is valid in $\mathsf{Ab}_{\pi'}$.

It remains to prove that the pseudoidentity $uw=v$ is valid in   $\mathsf{Ab}_{(p_1^{\beta_1}\ldots p_{l+b}^{\beta_{l+b}})}$. Let
$$\beta=\max\{\beta_1,\ldots,\beta_{l+b}\}.$$
We have
\begin{equation*}
   \begin{aligned}
   \varphi(uw)&=\varphi(u)x_1+(p_1\ldots p_{l+b})^{\beta}\varphi((p_1\ldots p_{l+b})^{\omega-\beta}(ux_2-ux_1))\\
  &=\varphi(u)x_1+0=\varphi(u)x_1=\varphi(v).
   \end{aligned}
 \end{equation*}
 Hence, the pseudoidentity $uw=v$ holds in $\mathsf{Ab}_{(p_1^{\beta_1}\ldots p_{l+b}^{\beta_{l+b}})}$.
This proves the claim.
\end{proof}

\begin{corollary}\label{cor4}
  Let $B\in \mathrm{M}_{s\times t}(\mathbb{Z}_\pi^{\sigma})$ and $C\in (\mathbb{Z}_\pi^{\sigma})^s$. If the system of equations $BX=C$ has a solution in $\widehat{\mathbb{Z}}_{\pi}$, then it has a solution in $\mathbb{Z}_\pi^{\sigma}$.
\end{corollary}
\begin{proof}
  By Lemma \ref{lemma46}, for every entry $u_{ij}$ of $B$, there is a computable $c_{ij}\in \mathbb{Z}$ such that $c_{ij}u_{ij}\in\mathbb{Z}$ and for every prime number $p$ dividing $c_{ij}$, $p$ lies in $P_\pi$. Let $\{p_1,\ldots, p_l\}\subseteq P_\pi$ be the set of all prime number divisors of the natural number   $c=\displaystyle\prod_{i=1}^s\displaystyle\prod_{j=1}^tc_{ij}$. There are $\beta_i\in \mathbb{N}$ such that $\pi=p_1^{\beta_1}\ldots p_l^{\beta_l}\pi'$ where $\gcd(\pi',p_1\ldots p_l)=1$.  Since $cB$ is in $\mathrm{M}_{s\times t}(\mathbb{Z})$, there are matrixes  $L\in \mathrm{GL}_{s}(\mathbb{Z})$ and $R\in \mathrm{GL}_{t}(\mathbb{Z})$ such that $D=L\left(cB\right)R$ is
   a diagonal matrix with entries in $\mathbb{Z}$. As, by assumption, the system $BX=C$ has a solution $V\in \mathrm{M}_{t\times 1}(\widehat{\mathbb{Z}}_{\pi})$, the system $DY=L(cC)$ has the solution  $R^{-1}V\in \mathrm{M}_{t\times 1}(\widehat{\mathbb{Z}}_{\pi})$, because we have
   $$DR^{-1}V=L(cB)RR^{-1}V=L(cB)V=Lc(BV)=LcC.$$
    Since $D$ is a diagonal matrix with entries in $\mathbb{Z}$ and $L(cC)\in (\mathbb{Z}_\pi^{\sigma})^{s}$, it follows from the preceding theorem, that there are $y_1,\ldots,y_t\in \mathbb{Z}_\pi^{\sigma}$ such that $Y= \begin{bmatrix}
    y_1\\
    \vdots \\
    y_t
    \end{bmatrix}$
is a solution of the system $DY=L(cC)$. The entries of the matrix  $X_1=c^{\omega}RY$ are in $\mathbb{Z}_\pi^{\sigma}$ and the following pseudoidentities are valid in $\mathsf{Ab}_{\pi'}$:
\begin{equation}\label{eq3}
   \begin{aligned}
   BX_1&=B(c^{\omega}RY)=c^{\omega-1}(cBR)Y=c^{\omega-1}L^{-1}(L(cB)RY)\\
   &=c^{\omega-1}L^{-1}DY=c^{\omega-1}L^{-1}(LcC)=c^{\omega}C=C,
  \end{aligned}
\end{equation}
where the last pseudoidentity follows from  the fact that $\gcd(c,\pi')=1$ and Lemma \ref{lemma45}.

Now let $\varphi:\widehat{\mathbb{Z}}_\pi\rightarrow \mathbb{Z}/(p_1^{\beta_1}\ldots p_l^{\beta_l})\mathbb{Z}$ be the natural projection. Since the system $BX=C$ has a solution in $\widehat{\mathbb{Z}}_\pi$, the system of congruences
\begin{equation}\label{eq2}
  \varphi(B)X\equiv\varphi(C) \  \pmod{p_1^{\beta_1}\ldots p_l^{\beta_l}}
\end{equation}
 has a solution in $\mathrm{M}_{t\times 1}(\mathbb{Z})$ where we extend the homomorphism $\varphi$ to the matrixes by applying it entrywise. Let $X_2\in \mathrm{M}_{t\times 1}(\mathbb{Z})$ be a solution of the system \eqref{eq2}.

 Let
  $$W=X_2+\left(p_1\ldots p_l\right)^{\omega} \left(X_1-X_2\right)\in  \mathrm{M}_{t\times 1}( \mathbb{Z}_\pi^{\sigma}).$$
We claim that $W$ is a solution of the system   $BX=C$ in $\widehat{\mathbb{Z}}_\pi$. By  Corollary \ref{cor3}, it is enough to show that the pseudoidentity $BW=C$ holds  in both $\mathsf{Ab}_{\pi'}$ and $\mathsf{Ab}_{(p_1^{\beta_1}\ldots p_l^{\beta_l})}$.
The proof of the claim is immediate  from the following facts:
\begin{itemize}
  \item By Lemma \ref{lemma45}, the pseudoidentity $\left(p_1\ldots p_l\right)^{\omega}=1$ is valid in $\mathsf{Ab}_{\pi'}$
  \item The pseudoidentity $\left(p_1\ldots p_l\right)^{\omega}=0$ is valid in $\mathsf{Ab}_{(p_1^{\beta_1}\ldots p_l^{\beta_l})}$.\qedhere
\end{itemize}
 \end{proof}

\subsection{The closure of  rational subsets in
 \texorpdfstring{$\Omega_A^{\sigma}\mathsf{Ab}_\pi$} {TEXT}}
  In order to show that the pseudovariety $\mathsf{Ab}_\pi$ is  $\sigma$-reducible,  we need to compute the closure of  $\psi_{\mathsf{Ab}_\pi}(L)$ in the free $\sigma$-algebra  $\Omega_A^{\sigma}\mathsf{Ab}_\pi$ where $L$ is a rational subset of the free monoid generated by $A$.
Since the image of a rational subset under a homomorphism is a rational subset,
       we must  compute the the closure of a rational subset of the free commutative monoid generated by $A$.

Let $M$ be a monoid. Finite unions of subsets of $M$ of the form
$$ab_1^*\ldots b_r^*\ \ \ (r\geq1, \ a,b_1,\ldots,b_r\in M),$$
are said to be \emph{semilinear}.
\begin{theorem}\cite[Proposition 2.2]{Delgado:1998}\label{lemma47}
 The semilinear subsets of a finitely generated commutative
monoid are precisely the rational subsets.
\end{theorem}

For a rational subset $L$ of the free commutative monoid generated by $A$, we denote by $Cl_{\sigma,\mathsf{Ab}_\pi}(L)$ and  $Cl_{\mathsf{Ab}_\pi}(L)$  the closure of $L$ in ${\Omega}^{\sigma}_{A}\mathsf{Ab}_\pi$ and  $\overline{\Omega}_{A}\mathsf{Ab}_\pi$, respectively.

\begin{lemma}\label{lemma43}
  Let $\pi$ be an infinite supernatural number and let $L$ be a finite union of sets of the form $a+b_1\mathbb{N}+\cdots+b_l\mathbb{N}$. Then the closure of $L$ in $\overline{\Omega}_A \mathsf{Ab_\pi}$  is a finite union of sets of the form
$$a+b_1\widehat{\mathbb{Z}}_\pi+\cdots+b_l\widehat{\mathbb{Z}}_\pi.$$
\end{lemma}

\begin{proof}
 We know that the  closure  of the finite  union is the union of closures of its terms. As the addition is continuous and  $\overline{\Omega}_A \mathsf{Ab_\pi}$ is a compact space, the closure of a sum is the sum of the closures of its summand.  Hence,  $Cl_{\mathsf{Ab}_\pi}(L)$ is a finite union of sets of the form:
$$a+Cl_{\mathsf{Ab}_\pi}(b_1\mathbb{N})+\cdots+Cl_{\mathsf{Ab}_\pi}(b_l\mathbb{N}).$$
So, it is enough to show that $Cl_{\mathsf{Ab}_\pi}(b_i\mathbb{N})=b_i\widehat{\mathbb{Z}}_\pi$.
 Let $x\in Cl_{\mathsf{Ab}_\pi}(b_i\mathbb{N})$ and consider a sequence $\{x_n\}_n$  in $b_i\mathbb{N}$ converging  to $x$. The elements of the sequence are of the form $b_iy_n$ with $y_n\in \mathbb{N}$. Since ${\mathbb{N}}\subseteq\widehat{\mathbb{Z}}_\pi$ and $\widehat{\mathbb{Z}}_\pi$ is a compact space, the sequence $\{y_n\}_n$ has a subsequence $\{y_{n_k}\}_k$  converging to a point $y\in \widehat{\mathbb{Z}}_\pi$. Hence, we must have $x=b_iy$. So, $x$ is in $b_i\widehat{\mathbb{Z}}_\pi$.

Conversely, as the closure of a subsemigroup is again a subsemigroup, for every $y\in Cl_{\mathsf{Ab}_\pi}(b_i\mathbb{N})$  the sequence $\{(n!-1)y\}_n$ is in $Cl_{\mathsf{Ab}_\pi}(b_i\mathbb{N})$ converging to $(\omega-1)y$. Hence, $Cl_{\mathsf{Ab}_\pi}(b_i\mathbb{N})$ is a closed topological subgroup of $\overline{\Omega}_A\mathsf{Ab}_\pi$ containing $b_i$. Therefore, the closed topological subgroup $b_i\widehat{\mathbb{Z}}_\pi$  generated by $b_i$ is contained in $Cl_{\mathsf{Ab}_\pi}(b_i\mathbb{N})$.
\end{proof}

\begin{theorem}\label{theorem3}
  Let  $L$ be a finite union of sets of the form $a+b_1\mathbb{N}+\cdots+b_l\mathbb{N}$. Then the closure of $L$ in ${\Omega}_A^{\sigma} \mathsf{Ab_\pi}$  is a finite union of sets of the form:
$$a+b_1{\mathbb{Z}}_\pi^{\sigma}+\cdots+b_l{\mathbb{Z}}_\pi^{\sigma}.$$
\end{theorem}
\begin{proof}
We have
$$Cl_{\sigma,\mathsf{Ab}_\pi}(L)=Cl_{\mathsf{Ab}_\pi}(L)\cap \Omega_A^\sigma\mathsf{Ab_\pi}.$$
Let $v\in Cl_{\sigma,\mathsf{Ab}_\pi}(L)$. By the preceding lemma, $Cl_{\mathsf{Ab}_\pi}(L)$ is a finite union of sets of the form:
$$a+b_1\widehat{\mathbb{Z}}_\pi+\cdots+b_l\widehat{\mathbb{Z}}_\pi.$$
Hence,  there are $y_1,\ldots,y_l\in \widehat{\mathbb{Z}}_\pi$, such that
$$v=a+b_1y_1+\cdots+b_ly_l.$$
Let $BY=C$ be a system of equations  where $B$ is  the matrix whose columns are the vectors $b_1,\ldots, b_l$ and $C$ is the vector $v-a$ seen as a column matrix.
Since the system $BY=C$ has a solution in $\widehat{\mathbb{Z}}_\pi$, by Corollary \ref{cor4}, this system has a solution in $\mathbb{Z}_\pi^{\sigma}$. Hence, there are $y'_1,\ldots, y'_l\in \mathbb{Z}_\pi^{\sigma}$ such that
$$v=a+b_1y'_1+\cdots+b_ly'_l.$$
Therefore, $v$ lies in $a+b_1\mathbb{Z}_\pi^{\sigma}+\cdots+b_l\mathbb{Z}_\pi^{\sigma}$.

The reverse inclusion is trivial.
\end{proof}

\begin{corollary}\label{cor5}
Let $L$ and $K$ be semilinear subsets of the $A$-generated free commutative monoid. Then the following hold:
\begin{itemize}
  \item $Cl_{\sigma,\mathsf{Ab}_\pi}(L+K)=Cl_{\sigma,\mathsf{Ab}_\pi}(L)+Cl_{\sigma,\mathsf{Ab}_\pi}(K)$;
  \item $Cl_{\sigma,\mathsf{Ab}_\pi}(L^{+})=\left<L\right>_{\sigma}$ where $\left<L\right>_{\sigma}$ denotes the $\sigma$-subalgebra of $\Omega_A^{\sigma}\mathsf{Ab}_\pi$ generated by $L$.
\end{itemize}
 \end{corollary}

\begin{proof}
Suppose the semilinear subsets $L$ and $K$ of the commutative monoid are written as $\displaystyle\bigcup_{i=1}^n \left(a_{0,i}+a_{1,i}\mathbb{N}+\cdots+a_{n,i}\mathbb{N}\right)$ and $\displaystyle\bigcup_{i=1}^m \left(b_{0,i}+b_{1,i}\mathbb{N}+\cdots+b_{m,i}\mathbb{N}\right)$, respectively. It is enough to observe that
 \begin{equation*}
  \begin{aligned}
  L+K&=\displaystyle\bigcup_{i=1}^n \left(a_{0,i}+a_{1,i}\mathbb{N}+\cdots+a_{n,i}\mathbb{N}\right)+\displaystyle\bigcup_{j=1}^m \left(b_{0,i}+b_{1,i}\mathbb{N}+\cdots+b_{m,i}\mathbb{N}\right)\\
  &=\displaystyle\bigcup_{i=1}^m\displaystyle\bigcup_{j=1}^m \left(a_{0,i}+b_{0,i}+a_{1,i}\mathbb{N}+\cdots+a_{n,i}\mathbb{N}+b_{1,i}\mathbb{N}+\cdots+b_{m,i}\mathbb{N}\right).
  \end{aligned}
 \end{equation*}
Since the subsemigroup generated by the union of the subsets $X$ and $Y$ of the free commutative monoid is $X^{+}+Y^{+}$, we have
$$L^{+}=\sum_{i=1}^{n}a_{0,i}\mathbb{N}+a_{1,i}\mathbb{N}+\cdots+a_{n,i}\mathbb{N}.$$
Now the results follow from the preceding theorem.
\end{proof}

\subsection{The pseudovariety  \texorpdfstring{$\mathsf{Ab}_\pi$}{TEXT}
  is \texorpdfstring{$\sigma$} {TEXT}-reducible}
As $(\overline{\Omega}_A\mathsf{S})^1$ and $\overline{\Omega}_A\mathsf{Ab}_\pi$ are compact spaces, for every rational subset $L$ of the free monoid, we have
\begin{equation}\label{eq21}
  \psi_{\mathsf{Ab}_\pi}(Cl(L))=Cl_{\mathsf{Ab}_\pi}(\psi_{\mathsf{Ab}_\pi}(L)).
\end{equation}
 Hence, the set $\psi_{\mathsf{Ab}_\pi}(Cl(L))\cap {\Omega}_A^\sigma\mathsf{Ab}_\pi$ is a closed set in ${\Omega}_A^\sigma\mathsf{Ab}_\pi$. We show that $\psi_{\mathsf{Ab}_\pi}(Cl_\sigma(L))=\psi_{\mathsf{Ab}_\pi}(Cl(L))\cap {\Omega}_A^\sigma\mathsf{Ab}_\pi$, and therefore, the pseudovariety $\mathsf{Ab}_\pi$ is $\sigma$-full.
\begin{theorem}
Let $L$ be a rational subset of $\Omega_A\mathsf{S}$. Then the following property holds:
\begin{equation}\label{pro}
   u\in \psi_{\mathsf{Ab}_\pi}(Cl(L))\cap \Omega_A^{\sigma}\mathsf{Ab}_\pi\Rightarrow \psi_{\mathsf{Ab}_\pi}^{-1}(u)\cap\Omega_A^{\sigma}\mathsf{S}\cap Cl(L)\neq\emptyset.
  \end{equation}
\end{theorem}

\begin{proof}
If $L$ is a finite set, then we have $Cl(L)=L$ and since $L\subseteq  \Omega_A\mathsf{S}\subseteq \Omega_A^{\sigma}\mathsf{S}$, Property \eqref{pro} holds.

As any rational subset of the free semigroup can be obtained by taking a finite number of finite subsets of the free semigroup and applying the union, product and the plus operation $L\rightarrow L^{+}$  a finite number of times, it is enough to show that these operations preserve Property \eqref{pro}.

Suppose that $L_1$ and $L_2$ are rational subsets of $\Omega_A\mathsf{S}$ satisfy Property \eqref{pro}. First we  show that  $L_1\cup L_2$ and $L_1L_2$ satisfy this property. We have
\begin{equation*}
  \begin{aligned}
  \psi_{\mathsf{Ab}_\pi}(Cl(L_1\cup L_2))&=\psi_{\mathsf{Ab}_\pi}(Cl(L_1)\cup Cl(L_2))\\
  &=\psi_{\mathsf{Ab}_\pi}(Cl(L_1))\cup \psi_{\mathsf{Ab}_\pi}(Cl(L_2)).
  \end{aligned}
\end{equation*}
Hence, if $u\in \psi_{\mathsf{Ab}_\pi}(Cl(L_1\cup L_2))\cap\Omega_A^{\sigma}\mathsf{Ab}_\pi$, then at least one of the sets $\psi_{\mathsf{Ab}_\pi}^{-1}(u)\cap\Omega_A^{\sigma}\mathsf{S}\cap Cl(L_1)$  and  $\psi_{\mathsf{Ab}_\pi}^{-1}(u)\cap\Omega_A^{\sigma}\mathsf{S}\cap Cl(L_2)$ is nonempty and so, $\psi_{\mathsf{Ab}_\pi}^{-1}(u)\cap\Omega_A^{\sigma}\mathsf{S}\cap Cl(L_1\cup L_2)$ is nonempty.

 As $\psi_{\mathsf{Ab}_\pi}$ is a homomorphism, $\psi_{\mathsf{Ab}_\pi}(L_1)$ and $\psi_{\mathsf{Ab}_\pi}(L_2)$ are rational subsets. Hence, by Corollary \ref{cor5}, the following equalities hold:
\begin{equation}\label{eq13}
  \begin{aligned}
 \psi_{\mathsf{Ab}_\pi}(Cl(L_1L_2))\cap \Omega_A^{\sigma}\mathsf{Ab}_\pi&\stackrel{\eqref{eq21}}{=} Cl_{\mathsf{Ab}_\pi}(\psi_{\mathsf{Ab}_\pi}(L_1L_2))\cap \Omega_A^{\sigma}\mathsf{Ab}_\pi\\
 &= Cl_{\sigma,\mathsf{Ab}_\pi}( \psi_{\mathsf{Ab}_\pi}(L_1L_2))\\
 &= Cl_{\sigma,\mathsf{Ab}_\pi}( \psi_{\mathsf{Ab}_\pi}(L_1)+\psi_{\mathsf{Ab}_\pi}(L_2))\\
  &=Cl_{\sigma,\mathsf{Ab}_\pi}( \psi_{\mathsf{Ab}_\pi}(L_1))+Cl_{\sigma,\mathsf{Ab}_\pi}( \psi_{\mathsf{Ab}_\pi}(L_2)).
  \end{aligned}
\end{equation}
Let $u\in \psi_{\mathsf{Ab}_\pi}(Cl(L_1L_2))\cap \Omega_A^{\sigma}\mathsf{Ab}_\pi$. By \eqref{eq13}, there are
$$u_i\in Cl_{\sigma,\mathsf{Ab}_\pi}( \psi_{\mathsf{Ab}_\pi}(L_i))=\psi_{\mathsf{Ab}_\pi}(Cl(L_i))\cap \Omega_A^{\sigma}\mathsf{Ab}_\pi\ \ (i=1,2)$$
 such that $u=u_1+u_2$. By the
 induction hypotheses, there are $w_i\in \psi_{\mathsf{Ab}_\pi}^{-1}(u_i)\cap\Omega_A^{\sigma}\mathsf{S}\cap Cl(L_i)$. Let $w=w_1w_2\in \Omega_A^{\sigma}\mathsf{S}\cap Cl(L_1)Cl(L_2)\subseteq\Omega_A^{\sigma}\mathsf{S}\cap Cl(L_1L_2)$.  Then we have
 $$\psi_{\mathsf{Ab}_\pi}(w)=\psi_{\mathsf{Ab}_\pi}(w_1w_2)=\psi_{\mathsf{Ab}_\pi}(w_1)+\psi_{\mathsf{Ab}_\pi}(w_2)=u_1+u_2=u.$$
 Thus,  $w$ belongs to $\Omega_A^{\sigma}\mathsf{S}\cap Cl(L_1L_2)\cap \psi_{\mathsf{Ab}_\pi}^{-1}(u)$.

It remains to show that $L^{+}$ preserves  Property \eqref{pro}.  For this purpose, we use the following lemma.
\begin{lemma}\label{lemma2}\cite[Lemma 6.6]{Jorge:2002}\label{lemma48}
Let $L$ be a rational subset of $\Omega_A\mathsf{S}$. Then there is a (computable)
finite subset $Y$ of $\Omega_A^{\kappa}\mathsf{S}$ such that
\begin{enumerate}
  \item $Cl(Y^{+})\subseteq Cl(L^{+})$
  \item The subgroups  $\left<Y\right>$ and $\left<L\right>$ of the free group generated by $A$ are equal.  Consequently, the subgroups $\left<\psi_{\mathsf{Ab}_\pi}(Y)\right>$ and $\left<\psi_{\mathsf{Ab}_\pi}(L)\right>$ are equal.
\end{enumerate}
\end{lemma}
Let $L$ be a rational subset of the free semigroup satisfying Property \eqref{pro} and consider  $Y=\{w_1,\ldots,w_n\}\subseteq \Omega_A^{\kappa}\mathsf{S}$ such that $Y$ satisfies the properties in Lemma \ref{lemma2}.
     Applying Lemma \ref{lemma2} and observing the fact that  $Cl_{\mathsf{Ab}_\pi}(\psi_{\mathsf{Ab}_\pi}(L^{+}))=Cl_{\mathsf{Ab}_\pi}(\left<\psi_{\mathsf{Ab}_\pi}(L)\right>)$ we obtain that, to show that the rational subset $L^{+}$ satisfies Property \eqref{pro}, it is enough to show that $Y^{+}$ satisfies this property.

  By Theorem \ref{theorem3}, we have
 \begin{equation*}
  \begin{aligned}
\psi_{\mathsf{Ab}_\pi}(Cl(Y^{+}))\cap \Omega_A^{\sigma}\mathsf{Ab}_\pi&\stackrel{\eqref{eq21}}{=} Cl_{\mathsf{Ab}_\pi}(\psi_{\mathsf{Ab}_\pi}(Y)^{+})\cap \Omega_A^{\sigma}\mathsf{Ab}_\pi\\
&=Cl_{\mathsf{Ab}_\pi}(\{\psi_{\mathsf{Ab}_\pi}(w_1),\ldots,\psi_{\mathsf{Ab}_\pi}(w_n)\}^{+})\cap\Omega_A^{\sigma}\mathsf{Ab}_\pi\\
&=\psi_{\mathsf{Ab}_\pi}(w_1){\mathbb{Z}_\pi^{\sigma}}+\cdots+\psi_{\mathsf{Ab}_\pi}(w_n){\mathbb{Z}_\pi^{\sigma}}.
  \end{aligned}
  \end{equation*}
Let $u\in  \psi_{\mathsf{Ab}_\pi}(Cl(Y^{+}))\cap \Omega_A^{\sigma}\mathsf{Ab}_\pi$. There are $y_1,\ldots,y_n\in \mathbb{Z}_\pi^\sigma$ such that
 $$u=\psi_{\mathsf{Ab}_\pi}(w_1)y_1+\cdots+\psi_{\mathsf{Ab}_\pi}(w_n)y_n.$$
By Lemma \ref{lemma50}, $y_i$ can be written in  the following form:
 \begin{equation*}
\begin{aligned}
 & n_1^{\omega-k_1}a_1+ n_2^{\omega-k_2}a_2+\cdots+n_m^{\omega-k_m}a_m+a_0,
\end{aligned}
\end{equation*}
 with the following conditions:
\begin{enumerate}
  \item  for every $j$ ($1\leq j\leq m$), $n_j$ and $k_j$ are natural numbers;
  \item  for every $i$ ($0\leq i\leq m$), $a_{i}$ is an integer number;
  \item for every prime number $p$ dividing $n_j$ ($1\leq j\leq m$), $p\in P_\pi$.
\end{enumerate}
   Let $y'_i=n_1^{\omega-k_1}b_1+ n_2^{\omega-k_2}b_2+\cdots+n_m^{\omega-k_m}b_m+b_0\in \Omega_1^\sigma\mathsf{S}$ with the following conditions ($0\leq j\leq m$):
\begin{equation*}
\begin{aligned}
b_j&=
\begin{cases}
a_j,\ \ \ \ \ \ \ \ \ \ \ \ \   \text{if $a_j\geq0$}\\
(\omega-1)|a_j|, \ \ \text{if}\  a_j<0
\end{cases}
\end{aligned}
\end{equation*}
    Let $w=w_1^{y'_1}\ldots w_n^{y'_n}$. Since $w_i\in Cl(Y^{+})$ and the closure of a subsemigroup is again a subsemigroup, $w_i^{y'_i}$ is in  $Cl(Y^{+})\cap \Omega_A^{\sigma}\mathsf{S}$. As $ \psi_{\mathsf{Ab}_\pi}$ is a homomorphism, we have $\psi_{\mathsf{Ab}_\pi}(w)=u$.
Hence, $w$ belongs to $\psi_{\mathsf{Ab}_\pi}^{-1}(u)\cap Cl(Y^{+})\cap \Omega_A^{\sigma}\mathsf{S}$.
This proves the theorem.
\end{proof}

\begin{corollary}\label{corollary1}
The pseudovariety $\mathsf{Ab}_\pi$ is $\sigma$-full.
\end{corollary}
\begin{proof}
  By the preceding theorem, for every rational subset $L$ of $\Omega_A\mathsf{S}$, the subset  $\psi_{\mathsf{Ab}_{\pi}}(Cl(L))\cap \Omega_A^\sigma\mathsf{Ab}_\pi$ is contained in $\psi_{\mathsf{Ab}_{\pi}}(Cl_\sigma(L))$. The reverse  inclusion is trivial. So, the subset $\psi_{\mathsf{Ab}_{\pi}}(Cl_\sigma(L))$ is closed in  $\Omega_A^\sigma\mathsf{Ab}_\pi$ and therefore,  $\mathsf{Ab}_\pi$ is $\sigma$-full.
\end{proof}

To complete the proof of Theorem \ref{theorem5}, it remains to show that the pseudovariety $\mathsf{Ab}_\pi$ is weakly $\sigma$-reducible.

\begin{theorem}\label{theorem6}
  The pseudovariety $\mathsf{Ab}_\pi$ is weakly $\sigma$-reducible with respect to all the systems of equations of $\sigma$-terms without parameters.
\end{theorem}
\begin{proof}
Let $X$ be a finite set of variables ($|X|=k$). Consider  a system of equations of the form
\begin{equation}\label{1}
  u_i=v_i \ \ \ (i=1,\ldots,m),
\end{equation}
with constraints $L_x\subseteq \Omega_{A}\mathsf{S}$ and  $u_i,v_i\in\Omega_{X}^\sigma \mathsf{S}$.
Suppose that the system (\ref{1}) has a solution modulo  $\mathsf{Ab_\pi}$.
Hence, the system of equations
\begin{equation}\label{15}
  \psi_{\mathsf{Ab}_{\pi}}(u_i)+(\omega-1)\psi_{\mathsf{Ab}_{\pi}}(v_i)=0\ \ \ (i=1,\ldots,m),
\end{equation}
with constraints $ \psi_{\mathsf{Ab}_\pi}(x)\in  \psi_{\mathsf{Ab}_\pi}(Cl(L_x))$ has a solution in $\overline{\Omega}_A\mathsf{Ab}_\pi$.
 Since we take $u_i, v_i \in \Omega_X^\sigma \mathsf{S}$ and $\Omega_X^\sigma \mathsf{Ab}_\pi$ is a commutative algebra, the system \eqref{15} can be written as follows:
\begin{equation}\label{5}
  n_{i1} \psi_{\mathsf{Ab}_\pi}(x_{1})+n_{i2} \psi_{\mathsf{Ab}_\pi}(x_{2})+\cdots+n_{ik} \psi_{\mathsf{Ab}_\pi}(x_{k})=0 \ \ \ (i=1,\ldots,m \ \text{and}\ x_i\in X),
\end{equation}
with $n_{ij}\in \mathbb{Z}^\sigma_\pi$ under the constraints
\begin{equation}\label{6}
 \psi_{\mathsf{Ab}_\pi}(x_{i})\in \psi_{\mathsf{Ab}_\pi}(Cl(L_{x_i})) \ \ \ (i=1,\ldots,m ).
\end{equation}
By Lemma \ref{lemma43}, each constraint in~(\ref{6}) is a disjoint union of sets of the form
$$ a_{i0}+a_{i1}\widehat{\mathbb{Z}}_\pi+\cdots+a_{il}\widehat{\mathbb{Z}}_\pi,\ \ (a_{i0},\ldots ,a_{il}\in \mathbb{Z}^r).$$
So, every $ \psi_{\mathsf{Ab}_\pi}(x_{i})$ can be written as $a_{i0}+a_{i1}y_{i1}+\cdots+a_{il}y_{il}$ with ${y_{ij}}\in\widehat{\mathbb{Z}}_\pi$. By substituting  $ \psi_{\mathsf{Ab}_\pi}(x_{i})$ in the system of equations (\ref{5}), we obtain a linear system $BY=C$ with $B\in \mathrm{M}_{r\times lk}(\mathbb{Z}^\sigma_\pi)$, $Y$ a column vector of variables, and  $C\in (\mathbb{Z}^\sigma_\pi)^r$.
Since by the assumption  the system $BY=C$ has a solution in $\widehat{\mathbb{Z}}_\pi$, by Corollary \ref{cor4} it has a solution in $\mathbb{Z}^\sigma_\pi$. Hence, the system \eqref{1} has a solution $\delta:\overline{\Omega}_X\mathsf{S}\rightarrow\overline{\Omega}_A\mathsf{S}$ modulo $\mathsf{Ab}_\pi$ such that $ \psi_{\mathsf{Ab}_\pi}(\delta(x_i))\in Cl_{\sigma,\mathsf{Ab}_\pi}(L_{x_i})$. So, the pseudovariety $\mathsf{Ab}_\pi$ is weakly $\sigma$-reducible with respect to this system.
\end{proof}
\begin{corollary}
 Let $\pi$ be  a recursive supernatural number. Then the pseudovariety $\mathsf{Ab}_\pi$ is completely $\sigma$-reducible.
\end{corollary}
\begin{proof}
The result follows from Proposition \ref{proposition2}, Corollary \ref{corollary1}, and Theorem \ref{theorem6}.
\end{proof}
\section*{Acknowledgments}
This work is part of the author's  preparation of a doctoral thesis under the supervision of Prof. Jorge Almeida, whose advice is gratefully acknowledged.

It was partially supported by the FCT Docoral Grant with reference (SFRH/ BD/98202/2013).
 It was also partially supported by CMUP (UID/MAT/00144/ 2013), which is funded by FCT (Portugal) with national (MEC) and European structural funds (FEDER), under the partnership agreement PT2020.
\bibliographystyle{amsplain}
\providecommand{\bysame}{\leavevmode\hbox to3em{\hrulefill}\thinspace}
\providecommand{\MR}{\relax\ifhmode\unskip\space\fi MR }
% \MRhref is called by the amsart/book/proc definition of \MR.
\providecommand{\MRhref}[2]{%
  \href{http://www.ams.org/mathscinet-getitem?mr=#1}{#2}
}
\providecommand{\href}[2]{#2}

\end{document}